\documentclass[oneside, article]{memoir}

\usepackage{amsmath, amsthm, amsfonts, amssymb, mathtools, commath, hyperref, cancel, bm}
\usepackage{doi}
\usepackage[T1]{fontenc}
\usepackage{newpxtext}
\usepackage{newpxmath}
\usepackage[tracking, spacing]{microtype}
\microtypecontext{spacing=nonfrench}

\newtheorem{theorem}{Theorem}
\newtheorem{corollary}{Corollary}

\begin{document}
  \title{Dirty derivative stability in the frequency domain}
  \author{Simon Kuang}
  \maketitle

  \begin{abstract}
    Differentiation filters can be made proper by composition with a first-order low-pass filter at a desired bandwidth, resulting in a ``dirty derivative.''
    A stable closed-loop system that relies on output derivatives remains stable when the derivatives are replaced with dirty derivatives of sufficiently high bandwidth.
    I prove and generalize this fact by a frequency-domain argument.
  \end{abstract}

% The dirty derivative filter replaces \(s\) in a controller with a proper approximation \(\delta(\sigma) = \frac{\sigma s}{s + \sigma}\), where \(\sigma > 0\).

  Suppose we have a SISO LTI system in controller normal form of degree \(n\):
  \begin{align}
    % \dot x &= A x + b u, \quad y = c^tx; \notag
    % \intertext{where,}
    \dod{}{t} x_1 &= x_2 \notag\\
    \dod{}{t} x_2 = \dod[2]{}{t} x_1 &= x_3 \notag\\
    &\ldots \notag\\
    \dod[n]{}{t} x_1 &= q\del{\dod{}{t}} x_1 + u \notag\\
    \intertext{where \(q\) is some polynomial of degree \(n-1\).
    Then \(y = x_1\).
    Let \(Y\) and \(U\), respectively, be the Laplace transforms of \(y\) and \(u\) under homogeneous initial conditions.}
    s^n Y &= q(s) Y + U.\notag
    \intertext{Defining \(p(s) = s^n - q(s)\),}
    p(s) Y &= U.
    \intertext{We are given a polynomial \(k\) of degree \(n - 1\) representing full state feedback. Under the feedback law \(U = k(s) Y - V\), the transfer function }
    \frac{Y}{V}
    &= \frac{1}{p(s) - k(s)}
    \label{eq:known-stable-tf}
    \intertext{is stable, and we wonder whether \(Y/V\) remains stable when we replace \(k(s)\) with \(k(\delta)\) where
     }
    \delta &= \frac{\sigma s}{s +\sigma} \notag, \quad \sigma \in (0, \infty)
    \intertext{is the \emph{dirty derivative}.
    In other words, we are looking for a kind of continuity at \(\sigma = \infty\). To this end,
    projectivize the \(\sigma\) coordinate using \(\tau = \sigma^{-1}\).}
    \delta &= \frac{s}{\tau s + 1}
    \intertext{Now let us define \(G\), a deformation of \(Y/V\) where \(\tau = 0\) recovers \eqref{eq:known-stable-tf}.}
    G(s, \tau)
    &= \frac{1}{H(s, \tau)} \\
    H(s, \tau) &= p(s) - k(\delta(s, \tau))
    % \intertext{
    \end{align}
    Thus our question is the following: for sufficiently small \(\tau > 0\), is is true that \(H(s, \tau)\approx H(s, 0)\) in a meaningful way?
    Marchi, Fraile, and Tabuada~\cite{marchi2022dirty} find, by adapting the Lyapunov equation of the non-dirty controller, that stability is locally preserved.
    Our result is slightly more general.
      
    \begin{theorem}
      The zeros of \(H(s, \tau)\) are continuous in \(\tau\).
    \end{theorem}
    \begin{proof}
      The zeros of \(H(s, \tau)\) are the roots of the numerator of \(H(s, \tau)\), regarded as a polynomial in \(s\) whose coefficients are polynomials in \(\tau\).
      A polynomial's roots vary continuously in the coefficients.
    \end{proof}

    Therefore we have the following result, which can also be phrased in terms of \(\sigma = \tau^{-1}\).
    \begin{corollary}
      For all \(\epsilon > 0\) there exists \(\tau > 0\) such that the the poles of \(G(s, \tau)\) are within \(\epsilon\) of those of \(G(s, 0)\).
    \end{corollary}

    \chapter{Root locus interpretation}
    We previously proved that the zeros of \(H(s, \tau)\) do not go very far as \(\tau\) is infinitesimally varied: in fact the same argument holds for the locus of solutions to \(H(s, \tau) = z\), where \(z \in \mathbb{C}\).
    Now we find where they go.
    
    In this section we characterize some curves \(s(\tau)\) along which \(H(s(\tau), \tau)\) is constant.
    The total derivative of \(H\) must vanish.
      \begin{align}
    \dif H(s, \tau)
    &= p'(s) \dif s
    - k'(\delta) \delta_s \dif s
    - k'(\delta) \delta_\tau \dif \tau \\
    % \intertext{where the partial derivatives are given by}
    &= p'(s) \dif s
    - \frac{k'(\delta)}{\del{\tau s + 1}^2} \dif s
    + \frac{k'(\delta)s^2} {\del{\tau s + 1}^2} \dif \tau
    \label{eq:total-derivative}
    \intertext{Setting \(\dif H=0\),}
    p'(s) \dif s
    - \frac{k'(\delta)}{\del{\tau s + 1}^2} \dif s
    &= -\frac{k'(\delta)s^2} {\del{\tau s + 1}^2} \dif \tau ,\\
    \intertext{or}
    \sbr{\del{\tau s + 1}^2 p'(s) - k'(\delta)} \dif s
    &= -k'(\delta) \dif \tau. 
  \end{align}
  
  \begin{theorem}
    \(H(s, \tau)\) is constant on integral curves of
    \begin{align}
    \dod{s}{\tau}
    &= -\frac{k'(\delta)}{\del{\tau s + 1}^2 p'(s) - k'(\delta)}
  \end{align}
  \end{theorem}
  The denominator implies that for small \(\tau\), the locus \(s(\tau)\) has a bifurcation at stationary points of \(H(s, \tau)\) (regarded as a function of \(s\)).
  % To see why this ought to be the case, consider the root locus of \(y - x^2 - \tau\) and \(x - \tau\) as \(\tau\) increases from \(0\).
  
  % I believe this is more or less enough to show that for sufficiently small \(\tau\), the zeros of \(H(s, \tau)\) deform continuously in \(\tau\).
  % \begin{enumerate}
  %   \item We see from \eqref{eq:homotopy} that for sufficiently small \(\tau\), the root locus \(s(\tau)\) starting from a root of \(H(s, 0)\) is regular, and hence does not introduce new zeros,  when the root is of degree 1.
  %   By a density argument in the state space representation \((A, B, K)\) (or in the Zariski topology on \(k(s)\)?) this should hold for all stabilizing \(k\).
  %   \item The maximum modulus principle says that zeros of \(H\) cannot pop out of thin air on the complex plane, even though \(H(s, \tau)\) has a degree \(2n\) numerator when \(\tau > 0\).
  % \end{enumerate}
  % By picking a sufficiently small \(\tau\) we have picked a sufficiently large \(\sigma\), which concludes the wproof.

  \chapter{Bode/Nyquist interpretation}
  We saw that not only is \(H(s, \tau)\) continuous in \(\tau\); it is also differentiable in \(\tau\).
  From \eqref{eq:total-derivative} we read that
  \begin{align}
    \dpd{H(s, \tau)}{\tau}
    &= \frac{k'(\delta(s)) s^2}{\del{\tau s + 1}^2},
    \label{eq:H_partial_tau}
    \intertext{and in particular, the magnitude changes as }
    \dpd{}{\tau}
    \log \left\|H(s, \tau)\right\|^2
    &=
    \frac{H\pd{\overline H}{\tau} + \overline{H} \pd{H}{\tau}}{\left\|H\right\|^2} \notag \\
    &=
    \frac{
      2 \Re \sbr{\overline{H} k'(\delta) s^2}
    }{
      \del{\tau s + 1}^2 \left\|H\right\|^2
    }.
  \end{align}
  Additionally, \eqref{eq:H_partial_tau} shows that the Nyquist plot, and therefore the stability characteristics, of \(H(s, \tau)\) deform continuously in \(\tau\).
  \begin{enumerate}
    \item When \(s \approx 0\), \(\pd{H}{\tau} \approx k'(0)s^2\).
    \item As \(s \to \infty\), \(\pd{H}{\tau} \to k'(1)\).
  \end{enumerate}

  \bibliographystyle{abbrvurl}
  \bibliography{refs}

\end{document}